\documentclass[11pt]{article}

\usepackage[a4paper, margin=1in]{geometry}
\usepackage{amsfonts,amsmath,amssymb,amsthm,graphicx,fixmath,latexsym,color}
\usepackage{color}
\usepackage{xcolor}
\usepackage{hyperref,epsfig}
\usepackage{algorithmic}
\usepackage{algorithm}
\usepackage{xspace}

\providecommand{\remove}[1]{}

\newcommand{\D}{\mathcal{D}}

\newcommand{\F}{\mathcal{F}}
\newcommand{\B}{\mathcal{B}}
\newcommand{\T}{\mathcal{T}}

\newtheorem{theorem}{Theorem}[section]
\newtheorem{lemma}[theorem]{Lemma}
\newtheorem{proposition}[theorem]{Proposition}
\newtheorem{claim}[theorem]{Claim}
\newtheorem{corollary}[theorem]{Corollary}

\newtheorem{notation}[theorem]{Notation}
\newtheorem*{theorem*}{Theorem}
\newtheorem*{lemma*}{Lemma}
\newtheorem*{proposition*}{Proposition}
\newtheorem{observation}[theorem]{Observation}

\begin{document}

\title{Blockers for Triangulations of a Convex Polygon and a Geometric Maker-Breaker Game}

\author{Chaya Keller\thanks{Department of Mathematics, Ben-Gurion University of the NEGEV, Be'er-Sheva Israel. \texttt{kellerc@math.bgu.ac.il}. Research partially supported by Grant 635/16 from the Israel Science Foundation, by the Shulamit Aloni Post-Doctoral Fellowship of the Israeli Ministry of Science and Technology, and by the Kreitman Foundation Post-Doctoral Fellowship.}
\and
Yael Stein\thanks{Department of Computer Science, Ben-Gurion University of the NEGEV, Be'er-Sheva Israel. \texttt{shvagery@cs.bgu.ac.il}. Research partially supported by the Lynn and William Frankel Center for Computer Sciences.}
}

\date{}
\maketitle

\begin{abstract}
Let $G$ be a complete convex geometric graph whose vertex set $P$ forms a convex polygon $C$, and let $\mathcal{F}$ be a family of subgraphs of $G$. A \emph{blocker} for $\mathcal{F}$ is a set of edges, of smallest possible size, that contains a common edge with every element of $\mathcal{F}$. Previous works determined the blockers for various families $\mathcal{F}$ of non-crossing subgraphs, including the families of all perfect matchings, all spanning trees, all Hamiltonian paths, etc.

In this paper we present a complete characterization of the family $\B$ of blockers for the family $\mathcal{T}$ of \emph{triangulations} of $C$.
In particular, we show that $|\B|=F_{2n-8}$, where $F_k$ is the $k$'th element in the Fibonacci sequence and $n=|P|$.

We use our characterization to obtain a tight result on a geometric Maker-Breaker game in which the board is the set of diagonals
of a convex $n$-gon $C$ and Maker seeks to occupy a triangulation of $C$. Namely, we show that in the $(1:1)$ triangulation game, Maker can ensure
a win within $n-3$ moves, and that in the $(1:2)$ triangulation game, Breaker can ensure a win within $n-3$ moves. In particular, the threshold bias for
the game is $2$.
\end{abstract}

%

\section{Introduction}

Let $G$ be a complete convex geometric graph, and let $\F$ be a family of subgraphs of $G$. We say that a set $B$ of edges is a \emph{blocking set} for $\F$ if it contains an edge in common with every element of $\F$. A blocking set for $\F$ of minimal size is called a \emph{blocker} for $\F$.

Determining the size of the blockers for $\F$ is a natural Tur\'{a}n-type question, as it is equivalent to determining the maximal size of a convex geometric graph that is \emph{free of $\F$} (i.e., does not contain an element of $\F$). This question was studied for various families $\F$, e.g., all sets of $k$ disjoint edges~\cite{KP96,K79} and all sets of $k$ pairwise crossing edges~(\cite{CP92}, and see also~\cite{BKV03}).

The most satisfactory answer for the `blockers' question is not only determining their size, but rather giving a complete characterization of the set of blockers. Such a characterization has been obtained for quite a few families of simple (i.e., non-crossing) graphs, including the family $\mathcal{M}$ of all simple perfect matchings in~\cite{KP12}, the family $\mathbb{T}$ of all simple spanning trees in~\cite{Hernando}, the family $\mathcal{H}$ of all Hamiltonian paths in~\cite{KP16}, etc. The characterizations gave rise to interesting classes of examples, including caterpillar graphs (see~\cite{Caterpillar,KP12}), combs (see~\cite{KPRU13}) and semi-simple perfect matchings (see~\cite{KP13}), and had applications to the structure of the `flip graphs' of the respective structures (see~\cite{Hernando,HHMMN99}).

In this paper we consider blockers for the family $\T$ of \emph{triangulations} of $G$. Triangulating a polygon is a central tool in computational geometry, used in numerous proofs and algorithms. In the special case of convex polygons, triangulations were studied from various points of view, such as finding the optimal triangulation w.r.t. different criteria (see ~\cite{KT06,Kli80}) and studying the `flip graph' $\mathcal{T(G)}$ of triangulations (see~\cite{HN99,Lee89}), whose properties are related to deep results in hyperbolic geometry, as shown in the seminal paper of Sleator, Tarjan, and Thurston~\cite{STT88}. For more on triangulations of a convex polygon, see the textbook~\cite{DRS10}.

We present a complete characterization of the blockers for $\T$. In order to present our result, we need a few notations.

Let $G$ be the complete geometric graph on a set $P$ of $n$ vertices, realized in the plane such that $P$ is the vertex set of a convex polygon $C$.
We label the vertices of $P$ cyclically (clockwise) by $0,1,\ldots,n-1$.

\begin{theorem}\label{thm:main}
Any blocker $B$ of $G$ is (up to cyclical rotation of $P$) of the type $B=B_1\cup B_2$, where
\begin{eqnarray*}
	 B_1 &=& \{ (0,2), (1,3), (2,4), \dots , (m,m+2)\};\\
	 B_2 &=& \{ (m+3,i_1), (m+4,i_2), (m+5,i_3), \dots , (n-1,i_{n-3-m})\}
\end{eqnarray*}
for some $1\leq m\leq n-3$, $1\leq i_j\leq m+1$, such that if $|i_j-i_k|\geq 2$ then the diagonals $(m+j+2,i_j)$ and $(m+k+2,i_k)$ do not cross.
\end{theorem}
In words, the theorem states that each blocker consists of two sets of edges. The first is a sequence of $m+1$ consecutive `ear-covers' (edges connecting two vertices of distance $2$) which cover the path $\langle 0,1,\ldots,m+2 \rangle$ on the boundary of $C$. The second is a set of $n-3-m$ leaf edges that connect each of the vertices $m+3,m+4,\ldots,n-1$ to an internal vertex of the path $\langle 0,1,\ldots,m+2 \rangle$, such that two edges whose endpoints on the path are not consecutive do not cross each other.\footnote{Very recently, it has come to our attention that Theorem~\ref{thm:main} was independently proved by Ali et al.~\cite{ACTK16}.}


We note that unlike blockers for perfect matchings and for simple (i.e., non crossing) spanning trees, the blockers for $\T$ are not simple.
However, each blocker can be represented as the union of two blockers for simple spanning trees on complementary subsets $P_1, P_2$ of $P$.
Indeed, as proved in~\cite{KPRU13}, any blocker for simple spanning trees of $P_i$ is a simple spanning caterpillar whose spine lies on the boundary of $\mathrm{conv}(P_i)$.
Any blocker for $\T$ is a union of two such caterpillars, whose spines form the interlacing sequences $\langle 0,2,4,\ldots,\rangle$ and $\langle 1,3,5,\ldots,\rangle$ with the induced leaf edges.

An example of a blocker is presented in Figure~\ref{fig:theorem}.

\begin{figure}
	\centering
	\includegraphics[width=.45\columnwidth]{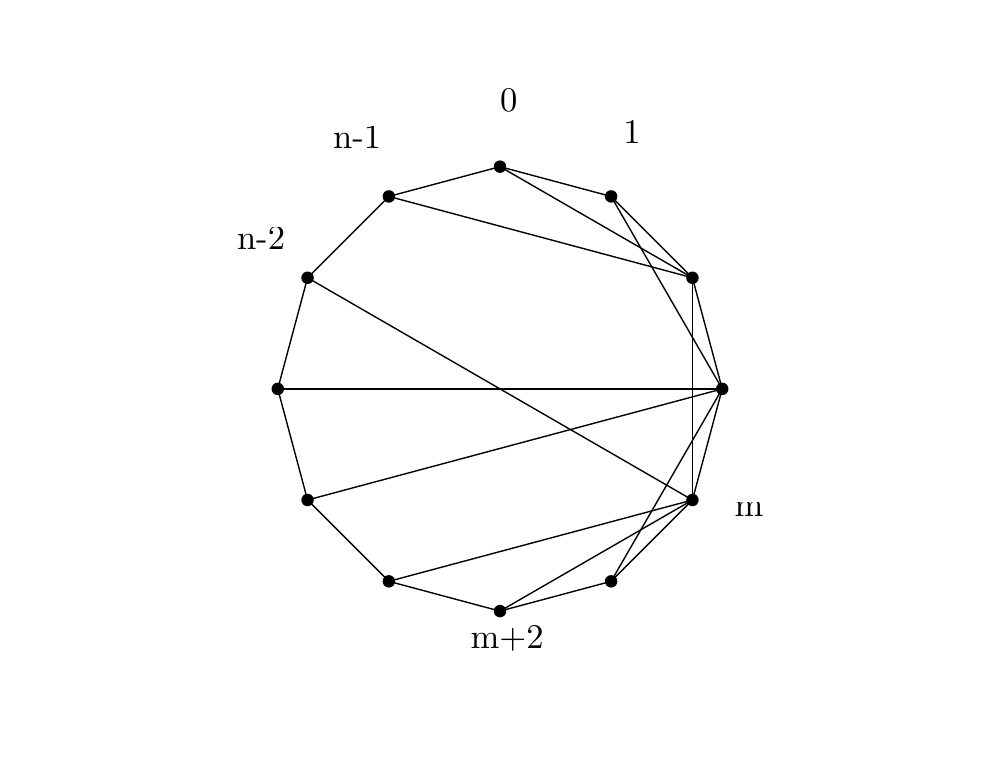}				
	\caption{An example of a blocker, with $n=12$ and $m=4$.} 
\label{fig:theorem}
\end{figure}

\medskip As a consequence of Theorem~\ref{thm:main}, we are able to calculate exactly the number of blockers.
\begin{theorem}\label{thm:count}
The number of blockers of $G$ (up to rotations) is $F_{2n-8}$, where $F_k$ is the $k$'th element in the Fibonacci sequence.
\end{theorem}

We apply our characterization to obtain a sharp result for a natural geometric Maker-Breaker game. Recall that in the
biased $(m:b)$ Maker-Breaker game on a board $X$ with respect to a hypergraph $\F \subset \mathrm{P}(X)$, the first player (Maker)
claims $m$ unoccupied elements $v \in X$ in each turn, and the second player (Breaker) answers by claiming $b$ vertices. Maker wins if the set of vertices
he occupied contains a winning set $S \in \F$, and otherwise, Breaker wins. The \emph{threshold bias} of the game is the minimal $b$ such that
Breaker wins the $(1:b)$ game. The study of Maker-Breaker games was initiated by Erd\H{o}s and
Selfridge~\cite{ES73} in 1973, and has expanded tremendously in the last few years (see the survey~\cite{K14}).

We consider a \emph{triangulation Maker-Breaker game}, in which the board is the set $X$ of diagonals of a convex $n$-gon $C$ and the
winning sets are the triangulations of $C$. We show the following:
\begin{theorem}
Let $C$ be a convex $n$-gon, $n \geq 5$. In the $(1:1)$ triangulation Maker-Breaker game on $C$, Maker has a winning strategy within $n-3$ moves. On the other hand, in the $(1:2)$ triangulation Maker-Breaker game, Breaker has a winning strategy within $n-3$ moves. In particular, the threshold bias of the game is $2$.
\end{theorem}

\medskip

The rest of this paper is organized as follows. In Section~\ref{sec:notations} we present definitions and notations that will be used in the sequel. The basic observations behind our proof are presented in Section~\ref{sec:observations}, and the proof of Theorem~\ref{thm:main} is presented in Section~\ref{sec:proof}. We prove Theorem~\ref{thm:count} in Section~\ref{sec:count}, and present the application to Maker-Breaker games in Section~\ref{sec:mb}.

\section{Definitions and Notations}
\label{sec:notations}

%


For any graph $G$, the \emph{degree} of a vertex $v\in V(G)$, $\deg(v)$, is the number of edges that emanate from $v$. The degree of $v$ \emph{with respect to a subgraph $B$} is denoted by $\deg_B(v)$.

Throughout the paper, $P$ will denote a set of $n$ points in a convex position in the plane, forming a $n$-gon $C$, and labelled cyclically clockwise from $0$ to $n-1$.
$G$ is the complete geometric graph on $P$.
All the operations on the index set $\{0,\dots,n-1\}$ are modulo $n$.

The \emph{order} of an edge $e=(i,j)\in E(G)$, denoted by $o(e)$, is $o(e)=min\{|i-j|, n-|i-j|\}$, and so, the edges of the $n$-gon $C$ are all of order $1$.
A \emph{diagonal} of $C$ is an edge of $G$ of order $\geq 2$. We denote by $\D(C)$ the set of diagonals of $C$.
A diagonal of $C$ of order $2$ is called an \emph{ear-cover}. We say that the ear-cover $(i-1,i+1)$ \emph{covers} the vertex $i$.
We say that two edges $e_1,e_2\in E(G)$ \emph{cross} if they share an interior point.

A \emph{triangulation} of $C$ is a subgraph $T$ of $G$ such that $E(T)$ consists of a maximal (with respect to inclusion) pairwise non-crossing set of diagonals of $C$. Any triangulation of $C$ contains $n-3$ diagonals of $C$. A blocking set $B$ for triangulations in $C$ is a subgraph of $G$ which contains a common edge with each element of $\T$. A blocking set with the minimum possible number of edges is called a blocker for triangulations in $C$, or in short, a \emph{blocker}. Sometimes we will abuse notation and identify $T$ and $B$ with the sets of their edges.

We denote by $C\setminus \{i\}$ the polygon obtained from $C$ by deleting the vertex $i$ and adding the edge $(i-1,i+1)$,
and by $B\setminus \{i\}$ the restriction of a blocker $B$ to the polygon $C\setminus \{i\}$ (i.e, $B\setminus \{i\}=B\setminus \{(i,j)\in E(B) | 0\leq j\leq n-1\}$ ).
%
%
%
%

Two canonical examples of blocking sets are:

\begin{enumerate}
	\item A `sun' -- a collection of all diagonals of $C$ that emanate from a fixed vertex $i\in P$, along with the ear-cover $(i-1,i+1)$ (see Figure~\ref{fig:sunnet}(a)).
	\item \label{ex:net} A `boundary net' -- a collection of $n-2$ consecutive ear-covers (see Figure~\ref{fig:sunnet}(b)).
\end{enumerate}

\begin{figure}
 \centering
    \centering
          \includegraphics[width=.45\columnwidth]{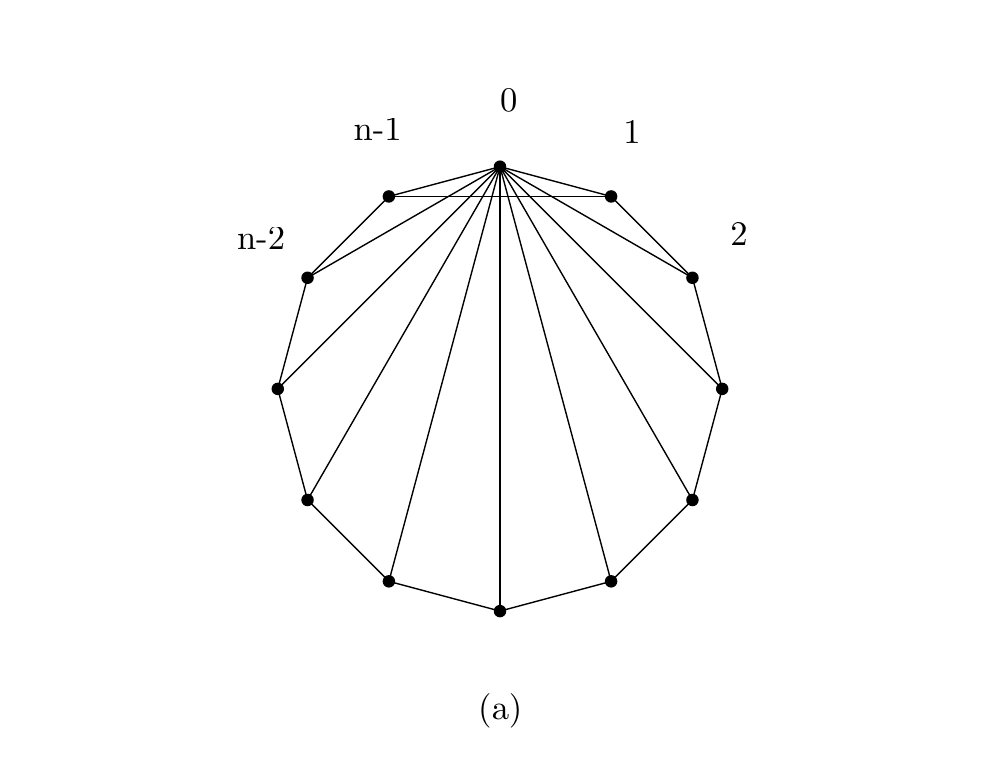}				
 \hspace{0.3cm}
    \centering
        \includegraphics[width=.45\columnwidth]{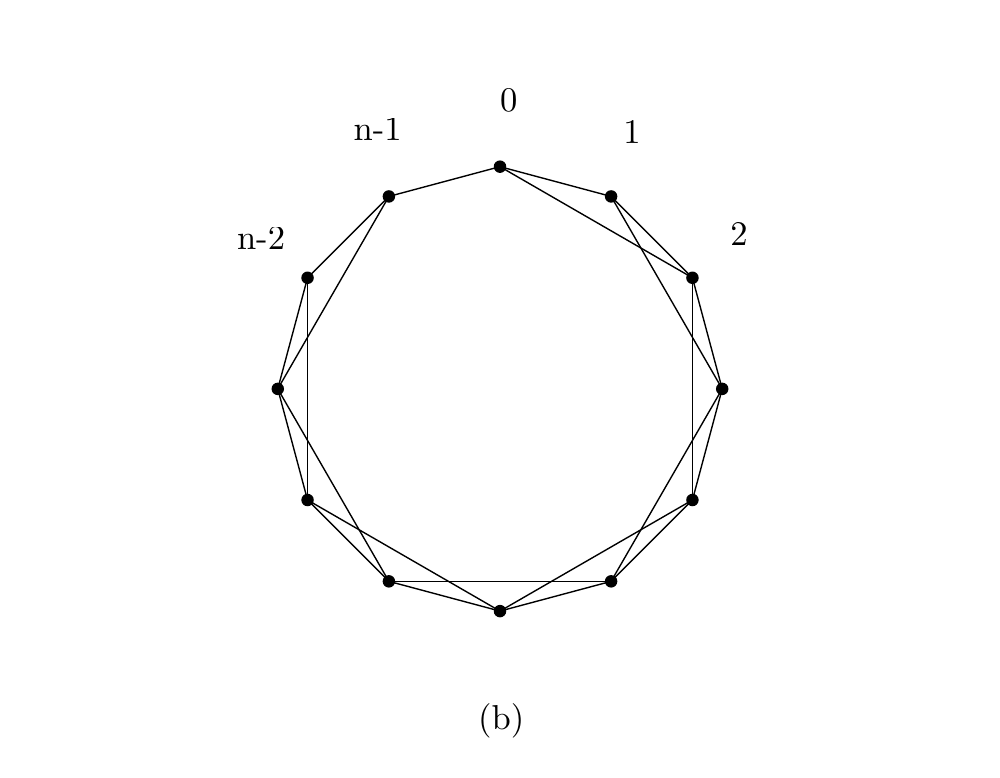}
 \caption{The two canonical examples of blockers: (a) is a `sun' and (b) is a `net'.}
\label{fig:sunnet}
\end{figure}

Clearly, the `sun' is a blocking set.
The `net' is indeed a blocking set, since any triangulation of $C$ contains at least two non crossing ear-covers.
In Theorem~\ref{thm:main} we present a full characterization of the blockers, proving that any blocker is, in some sense, a hybrid of these two canonical blocking sets.


%


In the notations of Theorem~\ref{thm:main}, we call the set $B_1$ \textit{the boundary net of $B$},
and the set $B_2$ \textit{the beams of $B$}.
The vertices $1,2,\dots ,m+1$ will be called \textit{interior vertices of the boundary net of $B$}.
We also say that a beam $(m+j+2,i_j)$ \textit{emanates from $i_j$}, where $i_j$ is its endpoint that is an interior vertex of the boundary net of $B$.


\section{Observations}
\label{sec:observations}

In this section we present a sequence of observations that will be used in the proof of our main theorem.



\begin{observation}\label{obs:isolated}
No blocker $B$ contains an isolated vertex.
\end{observation}

\begin{proof}
If $B$ contains an isolated vertex $i$, then it misses the star triangulation $T$ which consists of all diagonals that emanate from $i$.
\end{proof}

\begin{observation}\label{obs:size}
The size of each blocker is $n-2$.
\end{observation}

\begin{proof}
As the `sun' and the `net' blocking sets presented above consist of $n-2$ edges, it is sufficient to prove that any blocker has at least $n-2$ edges.
We will prove this by induction on $n=|P|$.

For $|P|=n=4$, any blocker must contain both diagonals of $C$.

For $n>4$, assume that any blocker $B'$ for a convex polygon of size $n'< n$ satisfies $|E(B')|=n'-2$.
Let $C$ be a convex polygon of size $n$ and let $B$ be a blocker of $C$.
By the minimality of $B$, we know that $|B| \leq n-2 <n$,
and thus, among the $n$ ear-covers, there exists an ear-cover $(i-1,i+1)$, that is not contained in $B$.
This implies that $B\setminus \{i\}$ is a blocker for $C\setminus \{i\}$, as
otherwise, a triangulation of $C\setminus \{i\}$ together with the edge $(i-1,i+1)$ forms a triangulation of $C$ that misses $B$ (i.e. has no common edge with $B$), a contradiction.
By the induction hypothesis, we have $|B\setminus \{i\}| \geq n-3$.
Finally, since $B$ does not have isolated vertices by Observation~\ref{obs:isolated}, we have $|B| \geq |B\setminus \{i\}|+1 \geq n-2$.
This completes the proof.
\end{proof}


\begin{observation}\label{obs:deg2}
If $B$ is a blocker and $i\in V(B)$ satisfies $deg_B(i)\geq 2$, then $(i-1,i+1)\in E(B)$.
\end{observation}

\begin{proof}
Assume to the contrary that $(i-1,i+1)\notin E(B)$.
By the assumption, we have $|B\setminus\{i\}|=(n-2)-\deg_B(i) \leq n-4$,
and thus, by Observation~\ref{obs:size}, $B\setminus\{i\}$ is not a blocker for triangulations in $C\setminus \{i\}$.
Hence, let $T'$ be a triangulation of $C\setminus \{i\}$ such that $T'\cap (B\setminus\{i\})=\emptyset$.
Let $T=T' \cup \{(i-1,i+1)\}$. Then $T$ is a triangulation of $C$ that misses $B$,
a contradiction.
\end{proof}

\begin{observation}\label{obs:deg1}
If $B$ is a blocker and $(i,j)\in E(B)$, 
then either $i$ or $j$ (or both) is covered by an ear-cover in $B$.
\end{observation}

\begin{proof}
If $deg_B(i)\geq 2$ then by Observation~\ref{obs:deg2}, $(i-1,i+1)\in E(B)$ and we are done.
Otherwise, by Observation~\ref{obs:isolated} we have $\deg_B(i)=1$. Hence, if $(j-1,j+1)\notin E(B)$ then $B$ misses the triangulation
$T=\{ (i,k) | 0\leq k\leq n-1 \wedge k\notin \{i-1, i, i+1, j\}\} \cup \{ (j-1,j+1)\}$ (see Figure~\ref{fig:deg1}), a contradiction.

\begin{figure}
	\centering
	\includegraphics[width=.45\columnwidth]{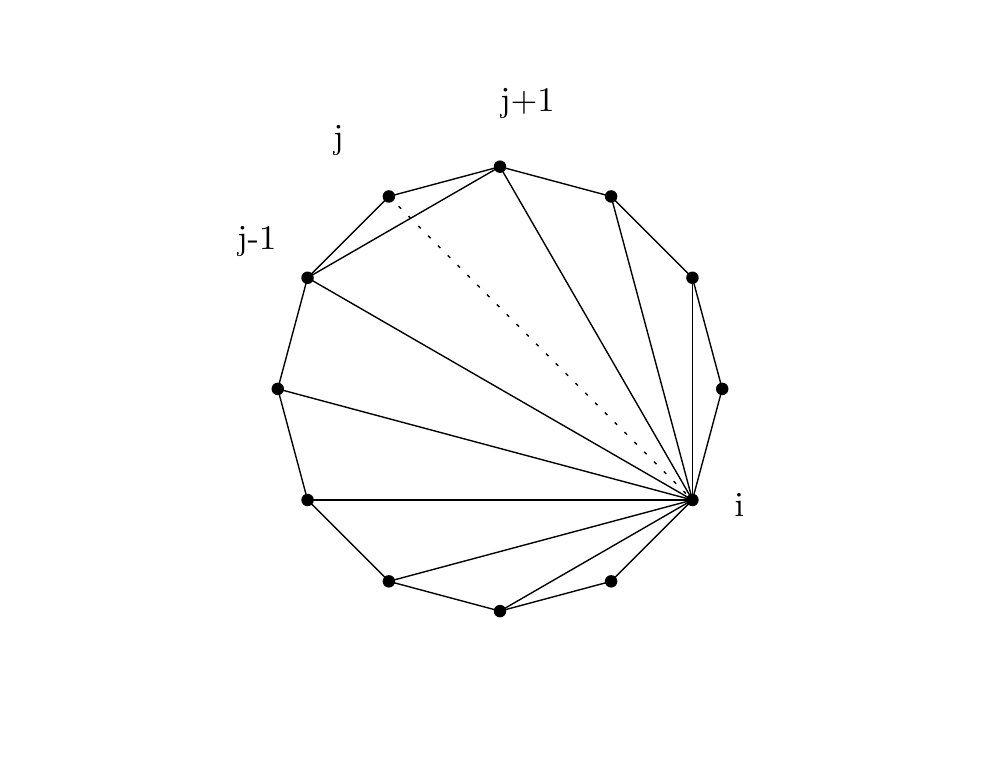}				
	\caption{The bold lines constitute a triangulation of $C$ that can be constructed when $deg_B(i)=deg_B(j)=1$ and $(j-1,j+1)\notin B$.}
\label{fig:deg1}
\end{figure}
\end{proof}

\begin{corollary}\label{Cor:1}
Any blocker contains at least two ear-covers.
\end{corollary}

\begin{proof}
Let $e=(i,j)\in E(B)$. By Observation~\ref{obs:deg1}, at least one of its endpoints is covered (in $B$) by an ear-cover, say $(i-1,i+1)\in E(B)$. If $j=i+2$ or $j=i-2$, we are done as $e$ itself is an ear-cover. Otherwise, $e$ is not an ear-cover, and by Observation~\ref{obs:deg1}, one of the endpoints of $(i-1,i+1)$ is covered by another ear-cover.
\end{proof}

\section{Proof of Theorem~\ref{thm:main}}
\label{sec:proof}

In this section we present the proof of Theorem~\ref{thm:main}.
Since the collection of blockers is invariant under rotations of $P$, we will describe the set of blockers up to these rotations.
We start with a characterization of the boundary nets of the blockers.

\begin{notation}
Let $B$ be a blocker.
The set of all ear-covers in $B$ is denoted by $$Ears(B)=\{e\in E(B) | o(e)=2\}.$$
\end{notation}

\begin{proposition}\label{prop:net}
Let $B$ be a blocker. There exists a vertex $a\in V(B)$ and an integer $1\leq m\leq n-3$, such that $Ears(B)=\{ (a,a+2),(a+1,a+3),\dots,(a+m,a+m+2)\}$.
\end{proposition}

\begin{proof}
First, by Corollary~\ref{Cor:1}, $Ears(B) \neq \emptyset$. If $i\in V(B)$ satisfies $(i-2,i),(i,i+2)\in Ears(B)$, then
$\deg_B(i)\geq 2$, and therefore, by Observation~\ref{obs:deg2}, $(i-1,i+1)\in Ears(B)$.
This implies that each connected component of $Ears(B)$ (in the topological sense in $\mathbb{R}^2$) is a set of consecutive
ear-covers, i.e., ear-covers emanating from consecutive vertices. We thus have to show that $Ears(B)$ is connected.

Assume to the contrary that $Ears(B)$ has at least two connected components, one of them with endpoints $w,x$ and
another with endpoints $y,z$. Assume that the cyclic order of $w,x,y,z$ on $P$ is $\langle w,x,y,z \rangle$ (see Figure~\ref{fig:2net}). By Observation~\ref{obs:deg2}, we have $\deg_B(w)=\deg_B(x)=\deg_B(y)=\deg_B(z)=1$.

\begin{figure}
	\centering
	\includegraphics[width=.35\columnwidth]{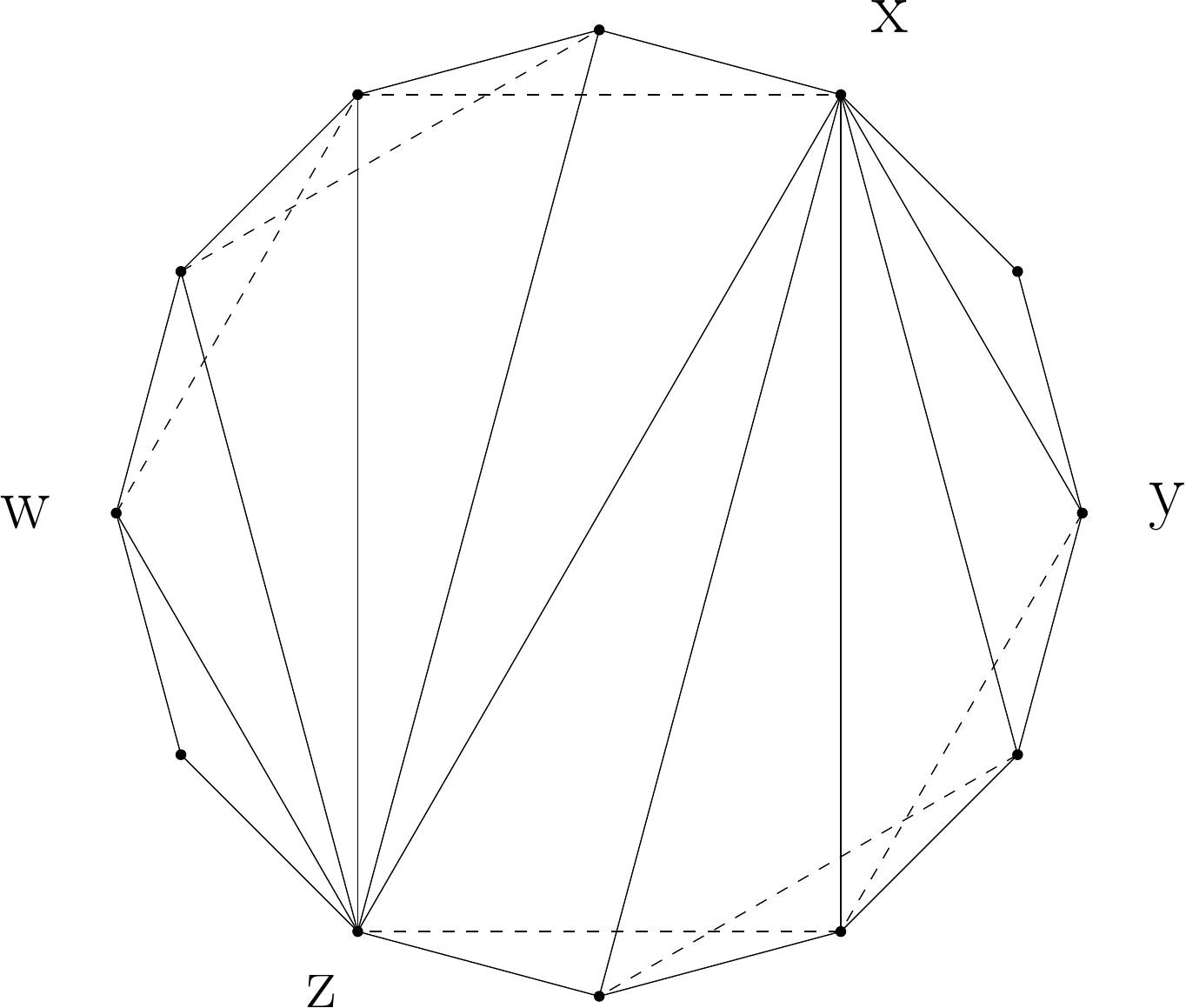}				
	\caption{An example of a triangulation of $C$ that can be constructed when there are two connected components in $Ears(B)$. 
					The triangulation is depicted in bold lines and $Ears(B)$ is depicted in dashed lines.}
	\label{fig:2net}
\end{figure}

Therefore, one can construct a triangulation $T$ in the following way: Connect $x$ to $z$, and to all the
vertices on the same side of $(x,z)$ like $y$. Then connect $z$ to all other vertices (that are on the same side of $(x,z)$ like $w$).
This construction implies that $B$ misses $T$,
a contradiction.
\end{proof}

Since we are interested in characterizing the blockers only up to rotation of the vertices, we can assume from now on that $a$ from Proposition~\ref{prop:net} is $0$.
In this way, Proposition~\ref{prop:net} gives us the description of the boundary net of a blocker, which is the set $B_1$ from Theorem~\ref{thm:main}.

We now prove that the remaining edges of a blocker are exactly the beams described in Theorem~\ref{thm:main}.
\begin{proposition}\label{prop:beam}
Let $B$ be a blocker with a boundary net $B_1=\{(0,2),(1,3),\dots,(m,m+2)\}$.
Then there exist $n-3-m$ integers $i_1,i_2,\dots,i_{n-3-m}\in \{1,2,\dots,m+1\}$ (possibly with repetitions),
s.t. $B\setminus B_1=\{ (m+3,i_1),(m+4,i_2),\dots,(n-1,i_{n-3-m})\}$,
and in addition, if $|i_j-i_k|\geq 2$ then the diagonals $(m+j+2,i_j)$ and $(m+k+2,i_k)$ do not cross.
\end{proposition}

\begin{proof}
The first part follows immediately from the observations presented above. Indeed, by Observation~\ref{obs:size} we have $|B\setminus B_1|=n-3-m$. By Observation~\ref{obs:deg1} and Proposition~\ref{prop:net} (with $a=0$ as discussed),
any edge in $B\setminus B_1$ emanates from one of the vertices ${1,...,m+1}$. On the other hand, by Observation~\ref{obs:isolated},
$B$ has no isolated vertices, and thus, each vertex in ${m+3,...,n-1}$ is incident to at least one (and thus, to exactly one) edge of $B$.


Now, let $(m+j+2,i_j),(m+k+2,i_k)\in B\setminus B_1$ be such that (w.l.o.g.) $i_j-i_k\geq 2$, and assume to the contrary that these edges cross each other. Let $x\in V(B)$ be such that $i_k< x< i_j$. (Note that such an $x$ exists since $i_j-i_k\geq 2$.) Construct a triangulation $T$ in the following way (demonstrated in Figure~\ref{fig:beam}): Connect the vertices ${x, m+j+2, m+k+2}$ to form a `central' triangle. Then, connect the vertex $m+j+2$ to all the vertices on the same side of $(m+j+2,x)$ as $i_k$,
and the vertex $m+k+2$ to all the vertices on the same side of $(m+k+2,x)$ as $i_j$. The vertices between
$m+k+2$ and $m+j+2$ can be all connected (arbitrarily) to $m+k+2$ too.

It is clear from the first part of this proof that $T$ misses $B$,
a contradiction.
\end{proof}

\begin{figure}
	\centering
	\includegraphics[width=.35\columnwidth]{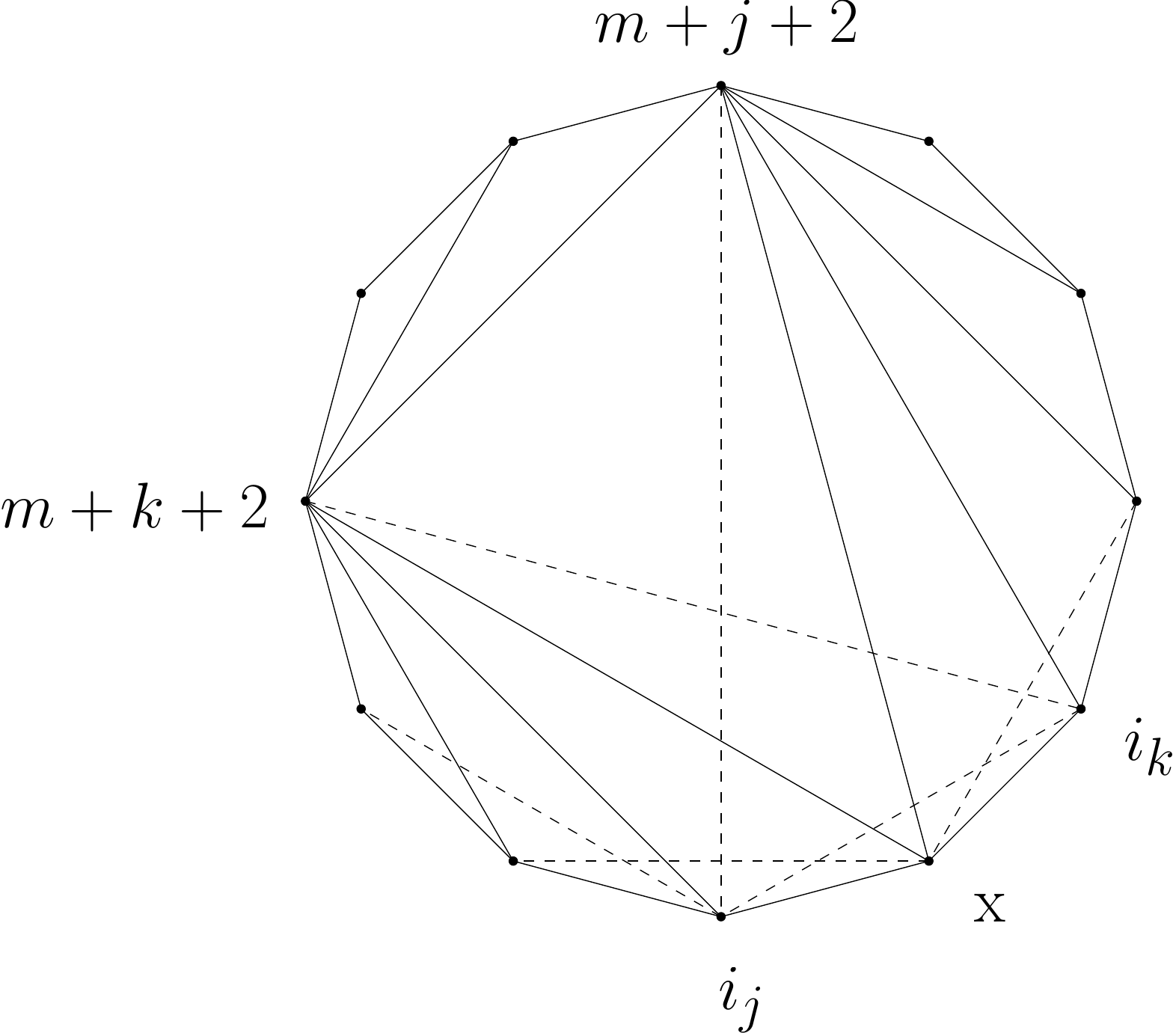}				
	\caption{A triangulation of $C$ that can be constructed when $i_j-i_k\geq 2$ and $(m+j+2,i_j)$ and $(m+k+2,i_k)$ cross.
						The dashed edges represent the edges of $B$ and the bold edges represent the triangulation.}
	\label{fig:beam}
\end{figure}



Propositions~\ref{prop:net} and~\ref{prop:beam} complete the proof of one direction of Theorem~\ref{thm:main}.
On the other hand, the following Proposition shows that any subgraph $B$ that satisfies the requirements of Theorem~\ref{thm:main} is indeed a blocker.

\begin{proposition}\label{prop:subgraph}
In the notations of Theorem~\ref{thm:main}, any subgraph $B$ of the type $B=B_1\cup B_2$ satisfies $E(B)\cap E(T)\neq\emptyset$ for any triangulation $T$ of $C$, and thus, is a blocker.
\end{proposition}

\begin{proof}
By induction on the size $n$ of $C$.\\
If $n=4$ then $B$ contains both diagonals of $C$ and thus meets any triangulation $T$ of $C$.


Suppose we proved the assertion for $n-1$ and let $C$ be a convex polygon of size $n$.
Let $B$ be a subgraph of the type $B=B_1\cup B_2$, as described in Theorem~\ref{thm:main}.
Assume to the contrary that there exists a triangulation $T$ of $C$ that misses $B$.
Since any triangulation contains an ear-cover, assume that $(i-1,i+1)\in T$ (and thus, $(i-1,i+1)\not \in E(B)$).
By the definition of $B_1,B_2$ it follows that $\deg_B(i)=1$, and thus, in the notations of Theorem~\ref{thm:main},
either $m+2 \leq i \leq n-1$ or $i=0$. Hence, $B\setminus \{i\}$ satisfies the requirements of Theorem~\ref{thm:main} w.r.t. the graph  $C\setminus \{i\}$.
Let $T'$ be the triangulation of $C\setminus \{i\}$ obtained from $T$ by omitting the diagonal $(i-1 ,i+1)$.
By the induction hypothesis, $B\setminus \{i\}$ is a blocker of $C\setminus \{i\}$.
As $T'$ is a triangulation of $C\setminus \{i\}$, this implies $E(B\setminus \{i\})\cap E(T')\neq\emptyset$.
A contradiction.
\end{proof}

Combining Propositions \ref{prop:net} -- \ref{prop:subgraph} together, completes the proof of Theorem~\ref{thm:main}.

\section{The Number of Blockers}
\label{sec:count}

For $n \geq 4$, we denote by $f(n)$ the number of blockers for triangulations of a convex $n$-gon, up to rotations. Recall that the Fibonacci sequence $\{F_n\}_{n=1}^{\infty}$ is defined by $F_1=F_2=1$ and $F_n=F_{n-1}+F_{n-2}$ for all $n \geq 3$. For sake of convenience, we set $F_0=1$ (note that this differs from the natural extension of the Fibonacci sequence).

In this Section we prove Theorem~\ref{thm:count}, namely, that for any $n \geq 4$ we have $f(n)=F_{2n-8}$.

%


\medskip Since $f(n)$ counts blockers up to rotations, we suppose w.l.o.g. that the boundary net of any blocker we consider starts at the vertex $0$ clockwise. Denote by $B_n^k$ the set of blockers whose boundary-net consists of $k$ ear-covers, and set $f^k(n)=|B_n^k|$. For $a,b\in \{0,1,2,\dots,n-1\}$ such that $a \leq b$ we denote by $[a,b]$ the set $\{a,a+1,a+2,\dots,b\}$. (If $a>b$ then $[a,b]=\emptyset$.)


\begin{figure}
	\centering
	\includegraphics[width=.35\columnwidth]{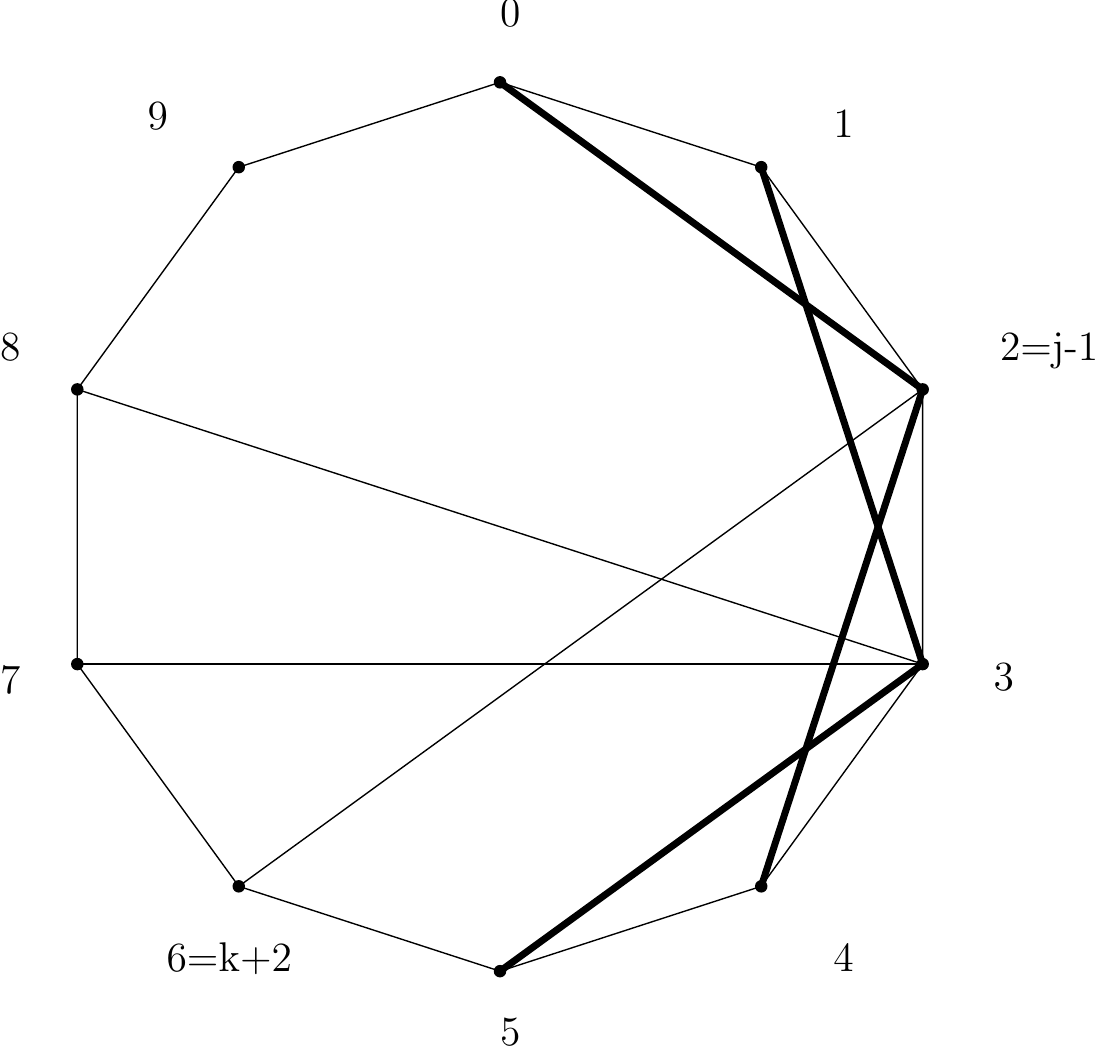}				
	\caption{An example of a part from a blocker in $B_n^{k,j,t}$, where $n=10$, $k=4$, $j=3$ and $t=2$.}
	\label{fig:count}
\end{figure}

%
%

We will use the following simple observations:

\begin{observation}\label{obs:fk}
For $n\geq 4$, $f(n)=\sum_{k=2}^{n-2}f^k(n)$.
\end{observation}

\begin{observation}\label{obs:fn1}
For $n\geq 4$, $f^{n-2}(n)=1$.
\end{observation}

\begin{observation}\label{obs:fibseq}
For $n\geq 4$, $$F_{2n}= \sum_{k=1}^n F_{2n-2k}.$$
\end{observation}

Observations~\ref{obs:fk},~\ref{obs:fn1} are trivial. Observation~\ref{obs:fibseq} is a well known property of the Fibonacci sequence that can be easily proved by induction.

The following Lemma is crucial for the proof of Theorem~\ref{thm:count}.

\begin{lemma}\label{lem:fk}
Let $n\geq 6$ and $2\leq k\leq n-4$. Then $$f^k(n)=\sum_{j=2}^k \sum_{i=j+2}^{n-1+j-k}f^j(i).$$
\end{lemma}

\begin{proof}
In any $B\in B_n^k$, the vertex $k+2$ is connected to one of the vertices $\{1,2,\dots,k-1\}$ (note that $k+2$ can be connected to $1$ since $k \leq n-4$).
For any $2\leq j\leq k$, let $B_n^{k,j}\subset B_n^k$ consist of the blockers in $B_n^k$ in which the vertex $k+2$ is connected to the vertex $j-1$.
We claim that $$|B_n^{k,j}|=\sum_{i=j+2}^{n-1+j-k}f^j(i).$$

In order to prove this, we further sub-divide $B_n^{k,j}$ as follows. For any $3 \leq j \leq k$ and $0 \leq t \leq n-3-k$, we let
\[
B_{n}^{k,j,t} := \{B \in B_{n}^{k,j}: (k+3,j),(k+4,j),\ldots,(k+t+2,j) \in B \wedge (k+t+3,j) \not \in B\}.
\]
For $j=2$, we use the same definition, with the exception (for $t=n-3-k$) $B_{n}^{k,2,n-3-k} = \{B \in B_{n}^{k,2}: (k+3,2),\ldots,(n-1,2) \in B\}$
(see Figure~\ref{fig:count}).

%
%
By Theorem~\ref{thm:main}, two beams of a blocker that emanate from non-consecutive vertices do not
cross. Hence, none of the vertices in $[k+3+t, n-1]$ is adjacent to any vertex whose index is greater than $j$.
Thus, there exists a bijection from $B_n^{k,j,t}$ to $B_{j+n-1-k-t}^{j}$ obtained by deleting the vertices $[j+1,k+1+t]$ and adding the edge $(j-1,k+t+2)$.
This implies that  $$ |B_n^{k,j,t}|=f^j(j+n-1-k-t).$$
Therefore,
\begin{align*}
|B_n^{k,j}| = \sum_{t=0}^{n-k-3}|B_n^{k,j,t}| = \sum_{t=0}^{n-k-3}f^j(j+n-1-k-t) = \sum_{i=j+2}^{n-1+j-k}f^j(i),
\end{align*}
where the last equality is obtained by changing the index of summation. Note that all terms $f^j(i)$ in this summation are positive, except for $f^2(5)=0$.

Thus, for $2\leq k\leq n-4$ we have
$$ f^k(n) = |B_n^k| = \sum_{j=2}^k \sum_{i=j+2}^{n-1+j-k} f^j(i) , $$
as asserted.
\end{proof}

Lemma~\ref{lem:fk} yields a recursive formula for $f^k(n)$, for all $2 \leq k \leq n-4$. The following claim allows to insert into the recursive formula the cases $k=n-3,n-2$.
\begin{claim}\label{Cl:2}
For all all $n \geq 5$, we have
\[
f^{n-3}(n)+f^{n-2}(n) = \sum_{j=2}^{n-3} f^j(j+2) = n-4.
\]
\end{claim}

\begin{proof}
The right hand side is equal to $n-4$, by Observation~\ref{obs:fn1}. As for the left hand side, $f^{n-3}(n)$ counts the blockers with $n-3$ ear-covers, in which the only vertex that is not contained in the boundary-net is $n-1$. This vertex can be connected to one of the $n-5$ vertices of $[2,n-4]$. (Recall that we assume that the boundary-net of the blocker starts in the vertex $0$, and thus, the vertex $n-1$ cannot be connected to the vertex $1$.) Hence, $f^{n-3}(n)=n-5$. Since $f^{n-2}(n)=1$ by Observation~\ref{obs:fn1}, the left hand side is equal to $n-4$, as asserted.
\end{proof}

Now we are ready to present the proof of Theorem~\ref{thm:count}.

%
%
%
%

\begin{proof}[Proof of Theorem~\ref{thm:count}]
Direct computations shows that the assertion of the theorem holds for small values of $n$. Hence, due to Observation~\ref{obs:fibseq}, it is sufficient to prove that
$$f(n)= \sum_{k=1}^{n-4} k f(n-k)$$.
Using Observation~\ref{obs:fk}, Lemma~\ref{lem:fk} and Claim~\ref{Cl:2},
we get
\begin{eqnarray*}
f(n) & = & \sum_{k=2}^{n-2}f^k(n) = \sum_{k=2}^{n-4}f^k(n) + (f^{n-3}(n) + f^{n-2}(n))\\
						 & = & \sum_{k=2}^{n-4} \sum_{j=2}^{k} \sum_{i=j+2}^{n-1+j-k} f^j(i) + \sum_{j=2}^{n-3}f^j(j+2)
                            = \sum_{k=2}^{n-3} \sum_{j=2}^{k} \sum_{i=j+2}^{n-1+j-k} f^j(i).
\end{eqnarray*}
Rearranging terms according to the index $i$ yields:
\begin{eqnarray*}
f(n) & = & \sum_{j=2}^{n-3}f^j(n-1) + 2\sum_{j=2}^{n-4}f^j(n-2) + 3\sum_{j=2}^{n-5}f^j(n-3) +\dots + (n-4)\sum_{j=2}^{2}f^j(4) \\
						 & = & f(n-1)+2f(n-2)+3f(n-3)+\dots +(n-4)f(4) = \sum_{k=1}^{n-4} k f(n-k).
\end{eqnarray*}
This completes the proof.
\end{proof}

\section{An Application to a Geometric Maker-Breaker Game}
\label{sec:mb}

Maker-Breaker games were introduced by Erd\H{o}s and Selfridge~\cite{ES73} in 1973. In the most common formulation of Maker-Breaker games (see~\cite{K14}), the parameters of the game are a \emph{board} $X$ (which is a finite set), a hypergraph $\F$ on $X$ whose elements are called \emph{winning sets}, and two integers $m,b$ which denote the numbers of moves of the players in each turn. In the $(m:b)$-biased Maker-Breaker game on $X$ with respect to $\F$, two players, called Maker and Breaker, take turns in alternately occupying a previously unoccupied vertices from $X$. Maker goes first and occupies $m$ vertices in each turn, and Breaker responses by occupying $b$ vertices in each turn. The $(1:1)$ game is called the \emph{unbiased} Maker-Breaker game, and the \emph{threshold bias} of the game is the minimal $b$ such that Breaker wins the $(1:b)$ game. Determining the threshold bias of a game is often considered the central goal in its study (see~\cite[Section~5]{K14}). Another goal is determining the minimal number of moves required for Maker (or Breaker) to secure a win (see~\cite{HKSS09}).


A central result in the study of unbiased Maker-Breaker games is the Erd\H{o}s-Selfridge Theorem~\cite{ES73} which states that if $\sum_{A \in \F}2^{-|A|}<\frac{1}{2}$ then Breaker has a winning strategy. Beck~\cite{Beck82} proved a generalization for $(m:b)$-games (sometimes called \emph{the biased Erd\H{o}s-Selfridge Theorem}) which states that if
\[
\sum_{A \in \F}(b+1)^{-\frac{|A|}{m}}<\frac{1}{b+1},
\]
then Breaker has an explicit winning strategy.

In this section we consider the \emph{triangulation Maker-Breaker game} in which the board $X$ is the set of diagonals of a convex $n$-gon $C$ and $\F$ is the family $\mathcal{T}$ of triangulations of $C$.\footnote{Naturally, for any family $\F$, characterization of \emph{blockers} with respect to $\F$ may help to supply a good strategy for Breaker, in a game where Maker's goal is to occupy an element of $\F$. Such strategies for general geometric Maker-Breaker games will be discussed in a separate paper; here we present only a specific result regarding the triangulation Maker-Breaker game.}


\begin{theorem}\label{thm:MB}
Consider the triangulation $(m:b)$-biased Maker-Breaker game played on a convex polygon on $n$ vertices.
\begin{enumerate}
	\item For all $n \geq 4$, in the $(1:1)$ game, Maker can ensure a win within $n-3$ moves.
    \item  For all $n \geq 5$, in the $(1:2)$ game, Breaker can ensure a win within $n-3$ moves. This statement holds also in a stronger version, in which Breaker occupies 2 vertices of $V$ only in his first move, and a single vertex in each other move.
\end{enumerate}
In particular, the threshold bias of the triangulation Maker-Breaker game for $n \geq 5$ is $2$.
\end{theorem}

Before we present the proof, a few remarks are due:
\begin{enumerate}
\item Obviously, the winning strategy in Part~(1) of the theorem, as well as in the stronger version of Part~(2), is the fastest possible.

\item In the proof of Part~(1), we show a slightly stronger statement: Maker wins even if Breaker makes the first move.

\item One can easily verify that the win of Breaker in the $(1:2)$ game (proved in Part~(2) of the theorem) does not follow from the biased Erd\H{o}s-Selfridge Theorem.
\end{enumerate}

\begin{proof}[Proof of Theorem~\ref{thm:MB}]
Consider first the $(1:1)$ game. We prove by induction on the number $n$ of vertices in the polygon $C$ that Maker wins, assuming that Breaker moves first. (Of course, there is no loss of generality in this assumption.) For $n=4$, the assertion is straightforward. Suppose that $n>4$ and that at the first move, Breaker occupies the diagonal $(x,y)$. Then, Maker occupies the diagonal $(x-1,x+1)$. By the induction hypothesis, in the game induced on the board $C \setminus \{x\}$, Maker has a winning strategy within $n-4$ moves. (Note that if Breaker occupies a diagonal that emanates from $x$, then Maker is able to occupy any ear-cover of $C \setminus \{x\}$.) Therefore, after $n-3$ moves, Maker completes occupying a triangulation and wins.

Now we consider a variant of the $(1:2)$ game in which Breaker makes two moves only after the first move of Maker. (Clearly, this implies that Breaker wins the standard $(1:2)$ game within the same number of moves, or even faster.)
After the first move of Maker, Breaker occupies two consecutive ear-covers $(a,a+2), (a+1,a+3)$, such that neither $a+1$ nor $a+2$ is an endpoint of the diagonal occupied by Maker. (This is always possible since $n \geq 5$; for $n=4$, Maker wins at the first move, regardless of the move she makes.) From now on, Breaker aims at constructing a blocker whose boundary-net is these two ear-covers. For any move of Maker in which she occupies $(x,a+1)$, Breaker answers by occupying $(x,a+2)$ (if this diagonal was not previously occupied), and vice versa -- if Maker occupies $(x,a+2)$, Breaker answers by occupying $(x,a+1)$ (if it was not previously occupied). For any other choice of Maker, Breaker chooses a vertex $x$ such that neither $(x,a+1)$ nor $(x,a+2)$ were previously occupied, and occupies the diagonal $(x,a+1)$.

It is clear that after $n-3$ turns Breaker occupies a blocker, for any choice of the moves of Maker. In particular, this implies that Maker is not able to occupy a triangulation in his $(n-3)$'th move, since any triangulation shares an edge with the blocker occupied by Breaker, and so cannot be fully occupied by Maker. This completes the proof.
\end{proof}

\subsection*{Acknowledgments}
The authors are grateful to Reuven Cohen for pointing out the relation between characterization of blockers and Maker-Breaker games, and for helpful suggestions.

\bibliographystyle{plain}
\bibliography{ref}

\end{document}